\documentclass[11pt]{amsart}
\usepackage[utf8]{inputenc}
\usepackage[english]{babel}
\usepackage{amsmath}
\usepackage{amsthm}
\usepackage{amsfonts}
\usepackage{amssymb}
\usepackage{esdiff}
\usepackage{mathtools} 
\usepackage{xcolor} 
\usepackage{bbm} 
\usepackage{dsfont} 
\usepackage{diagbox}
\usepackage{subcaption} 
\usepackage{pdfsync}
\usepackage{graphicx}
\usepackage{wrapfig}
\usepackage{amsthm} 

\newtheorem{theorem}{Theorem}
\newtheorem{proposition}{Proposition}

\newtheorem{lemma}{Lemma}
\newtheorem{conjecture}{Conjecture}

\theoremstyle{definition}
\newtheorem{definition}[theorem]{Definition}

\newcommand{\Vol}{\text{Vol}}
\newcommand{\Iso}{\text{Iso}}
\newcommand{\Isop}{\text{Isop}}
\newcommand{\N}{\mathbb{N}}

\newcommand{\R}{\mathbb{R}}

\newcommand{\conbod}{\mathcal{K}^n}

\newcommand{\Sph}{\mathbb{S}^{n-1}}
 \usepackage[bookmarks=true,
bookmarksnumbered=true, breaklinks=true,
pdfstartview=FitH, hyperfigures=false,
plainpages=false, naturalnames=true,
colorlinks=true,
pdfpagelabels]{hyperref}

\begin{document}

\title[Bezout for mixed volumes]{A New Excluding Condition Towards the Soprunov-Zvavitch Conjecture on Bezout-type inequalities}\footnote{

MSC Classification 2020: Primary 52A20; Secondary 52B60
}
\author{Maud Szusterman}
\address[Maud Szusterman]{Department of Mathematics\\ Universite de Paris\\ France}
\thanks{The author was supported by the National Science Foundation under Grant DMS-1929284 while the author was in residence at the Institute for Computational and Experimental Research in Mathematics in Providence, RI, during the Harmonic Analysis and Convexity program.}

\email{maud.szusterman@imj-prg.fr}

\maketitle

\begin{abstract}In 2015, I. Soprunov and A. Zvavitch have shown how to use the Bernstein-Khovanskii-Kushnirenko theorem to derive non-negativity of a certain bilinear form $F_{\Delta}$, defined on (pairs of) convex bodies. Together with C. Saroglou, they proved non-negativity of $F_K$ characterizes simplices, among all polytopes. It is conjectured the characterization further holds among all convex bodies. Towards this conjecture, several necessary conditions on $K$ (for non-negativity of $F_K$), were derived. We give a new necessary condition, expressed with isoperimetric ratios, which provides a further step towards a (conjectural) characterization of simplices among a certain subclass of convex bodies.
\end{abstract}

\section{Introduction}
\label{sec:intro}

The isoperimetric ratio of a convex body $K \subset \R^n$, is usually defined as $\frac{|\partial K|_{n-1}^n}{|K|_n^{n-1}}$, or as a power of this ratio. Here $|K|_n$ denotes the Lebesgue measure of $K$ (its volume), and $|\partial K|_{n-1}$ its surface area (the $(n-1)$-Hausdorff measure of its topological boundary). A classical result in convex geometry, the isoperimetric inequality, states that Euclidean balls minimize this ratio, and are the only minimizers. On the other side, the isoperimetric ratio can be arbitrarily large. However, note this ratio is not affine invariant. Following K. Ball (see \cite[Theorem 2]{Ba}), consider instead the affine invariant quantity:
$$Iso(K)=\min_{T\in O(n)} \frac{|\partial (TK)|_{n-1}^n}{|TK|_n^{n-1}},$$
where the minimum runs over the orthogonal group $O(n)$. Then $Iso(K)$ is upper bounded. In fact, denoting by $\Delta_n$ an $n$-simplex, Ball has proved :
$$Iso(K)\leq Iso(\Delta_n),$$
yielding a reverse isoperimetric inequality. Additionally, it was proved in \cite{Ba} that $n$-simplices are the only maximizers of $Iso(K)$.

Denote $V_n(L_1, ... , L_n)$ the mixed volume of $n$ convex bodies, and denote $V_n(K)=V(K[n])$ the volume of $K$. Since $|\partial K|_{n-1}=nV_n(K[n-1],B_2^n)$, where $B_2^n$ denotes the (unit) $l_2$-ball in $\R^n$, Ball's result can be reformulated within the language of mixed volumes as follows : for any convex body $K\subset \R^n$, there exists an ellipsoid $\mathcal{E}$, such that
$$V_n(K[n-1],\mathcal{E})^{n} V_n(\Delta_n)^{n-1} \leq  V_n(\Delta_n[n-1], \mathcal{E})^{n-1} V_n(K)^{n-1}.$$

Another interesting inequality of mixed volumes involving $\Delta_n$ was derived by I. Soprunov and A. Zvavitch (see \cite{SZ}), using the Bernstein-Kushnirenko-Khovanskii theorem (see \cite{Be}, \cite{Kho}, \cite{Kush}), and an inequality from real algebraic geometry called Bezout inequality:
$$\forall L_1, L_2 \in \mathcal{K}^n :\hspace{3mm} V_n(L_1,L_2, \Delta_n[n-2])V_n(\Delta_n) \leq V_n(L_1, \Delta_n[n-1])V_n(L_2, \Delta_n[n-1]).$$
This set of inequalities (called Bezout inequalities in \cite{SZ}) can be thought of as non-negativity of the bilinear form $F_{\Delta}$, defined as $$F_{\Delta}(A,B)=V_n(A, \Delta_n[n-1])V_n(B, \Delta_n[n-1])-
 V_n(A,B, \Delta_n[n-2])V_n(\Delta_n).$$
 
 A similar bilinear form $F_K$ can be defined for an arbitrary convex body $K$. Since $V_n(TL_1, ...,TL_n)=|det(T)|V_n(L_1, ... , L_n)$ (for any $n$-tuple of convex bodies $(L_i)$, and any affine transform $T$), non-negativity of $F_K$ is an affine invariant property.
 In $2018$, Saroglou, Soprunov and Zvavitch obtained a characterization of $n$-simplices among polytopes :
 \begin{theorem}
 \label{theorem:sszpolytopes}
 Let $P \subset \R^n$ be an $n$-polytope. Assume $F_P \geq 0$. Then $P$ is an $n$-simplex.
 \end{theorem}
 
 (here $F_P\geq 0$ means $F_P(A,B)\geq 0$ for all $(A,B)\in (\mathcal{K}^n)^2$, where $\conbod$ denotes the set of all convex bodies in $\R^n$).
 
 It was conjectured that this characterization further holds among all ($n$-dimensional) convex bodies, see \cite[Conj 1.2]{SZ}, or \cite[Conj 5.1]{SSZ2}.
 
 \begin{conjecture}
 \label{mainconj}
  Let $K \subset \R^n$ be a convex body. Assume $F_K \geq 0$. Then $K$ is an $n$-simplex.
 \end{conjecture}
 
 Several necessary conditions (for $F_K\geq0$ to hold) were derived : $K$ cannot be decomposable, nor weakly decomposable (see \cite[Def 5.3]{SSZ2} where this notion was introduced), the surface area measure $S_K$ cannot have a regular direction\footnote{a unit vector $u$ such that $K^u$ is $0$-dimensional, where $K^u=\{y\in K : \langle y,u \rangle=\max_{z\in K}\langle z,u \rangle \}$} in its support (see \cite[Prop 4.2]{SSZ2}), $K$ cannot have infinitely many facets on its boundary. We refer to \cite{SZ}, \cite{SSZ1}, \cite{SSZ2} for proofs, and comments on these conditions.
 
 In this work, we provide a new necessary condition, namely:
 \begin{theorem}
 \label{sevilla:main}
Let $K$ be a convex body such that $K$ has a facet $F$ satisfying :
$Isop(F)>Isop(K)$. Then $F_K\geq 0$ doesn't hold : there exists a pair of convex bodies $(A,B)\in (\mathcal{K}^n)^2$ such that $F_K(A,B)<0$.
 \end{theorem}

 It allows in particular to recover the necessity (for non-negativity of $F_K$) of having at most finitely many facets. We leave as an open question whether Theorem \ref{sevilla:main} allows to recover characterization amongst polytopes, i.e. Theorem \ref{theorem:sszpolytopes}, see Question 1 below.

{\bf Acknowledgements.} The author wishes to thank Artem Zvavitch for introducing her to this topic and for feedback on preliminary drafts of the current work. The author also thanks Evgueni Abakumov and Omer Friedland for many discussions and support. Finally, many thanks to Dylan Langharst for pointing to the reverse isoperimetric inequality by Ball, and for many editing suggestions which helped clarifying the presentation.

\section{Preliminaries}
\label{sec_notation}
A classical result due to Minkowski \cite{Mink} states that, for any $m\geq 2$, and any $m$-tuple $(K_1, ... , K_m)$ of compact convex sets in $\R^n$, the volume of the Minkowski sum $\sum \lambda_i K_i$, is a polynomial in the $\lambda_i \geq 0$. More precisely, Minkowski's theorem asserts there exists coefficients $c_a \geq 0$, indexed by $m$-tuples $a=(a_1, ... , a_m) \in \mathbb{N}^m$ summing to $n$ (i.e. $|a|=a_1+...+a_m=n$), which only depend on $(K_1, ... , K_m)$, and such that, for any $\lambda_1, ... , \lambda_m \geq 0$:
$$\left| \sum_{i=1}^m \lambda_i K_i\right|_n=\sum_{a\in \N^m , |a|=n} \frac{n!}{a_1 ! ... a_m!} \left( \prod_{j=1}^m \lambda_j^{a_j} \right) c_a =: \sum_{|a|=n} {n \choose a} \lambda^a c_a.$$
These coefficients are called mixed volumes, and usually one denotes $$V_n(K_1[a_1], ... , K_m[a_m]):=c_a,$$ with $a=(a_1, ... ,a_m)$, where $K_i[a_i]$ means that $K_i$ appears $a_i$ times as an argument.

It follows directly from Minkowski's theorem that, for any compact convex sets $K_1, K'_1, K_2, ... , K_n$, any $x_1\in \R^n$, any $\lambda \geq 0$, and any permutation $\sigma \in \mathcal{S}(n)$ :
\begin{itemize}
\item[i-] $V_n(K_1, ... , K_n)=\frac{1}{n!} \sum_{J \subset [n]} (-1)^{n-|J|} |K_J|_n$ where \footnote{with the convention $K_{\emptyset}=\emptyset$} $K_J=\sum_{i\in J} K_i$
\item[ii-] $V_n(K_1+x_1, K_2, ... , K_n)=V_n(K_1, ... , K_n)$  (translation invariance)
\item[iii-] $V_n(K_1+\lambda K'_1, K_2, ... , K_n)=V_n(K_1, K_2, ... , K_n)+\lambda V_n(K'_1, K_2, ... , K_n)$ (multilinearity)
\item[iv-] $V_n(K_{\sigma(1)}, ... , K_{\sigma(n)})=V_n(K_1, ... , K_n)$ (symmetry in the arguments)
\item[v-] $V_n(.)$ is continuous on $(\conbod)^n$, with respect to Hausdorff topology.
\end{itemize}

Moreover, Minkowski proved $V_n(.)$ also enjoys the following properties :
\begin{itemize}
\item[vi-] $V_n(K_1, ... , K_n)\geq 0$ (non-negativity)
\item[vii-] if $K_1 \subset K'_1$, then $V_n(K_1, K_2, ... , K_n) \leq V_n(K'_1, K_2,  ... , K_n)$ (monotonicity).
\end{itemize}

Hence $V_n(.)$ is a (multilinear) functional on $(\conbod)^n$. When the underlying dimension (i.e. the total number of arguments $V_n(.)$ takes) is clear, we may drop the subscript $n$, and write $V(.)$ rather than $V_n(.)$. It follows directly from the definition of mixed volumes, that $V_n(K[n])=|K|_n$. Therefore we may slightly abuse notation, and write $V_n(K)$, or even $V(K)$, instead of $V_n(K[n])$, as this shortcut seems common in the literature. To avoid confusion, we only use this shortcut when $K$ is indeed $n$-dimensional, i.e. when $K$ is a non-degenerate convex body in $\R^n$.

Let $u\in \mathbb{S}^{n-1}$ be a unit vector. Denote  $\pi_u$ the orthogonal projection onto $u^{\perp}$. If $u_1, ... , u_k$ are $k$ linearly independent unit vectors in $\R^n$, denote $\pi_U$ the orthogonal projection onto $(u_1, ... , u_k)^{\perp}$.
We will also need the following well-known property of mixed volumes :
\begin{itemize}
\item[viii-] $$V_n([0,u], K_2, ... , K_n)=\frac{1}{n}V_{n-1}(\pi_u K_2, ... , \pi_u K_n) ,$$
\item[ix-] \begin{align*}V_n([0,u_1], ... , &[0,u_k], K_{k+1}, ... , K_n)=\frac{k! V_k([0,u_1], ... , [0,u_k])}{n(n-1) ... (n-k+1)} V_{n-k}\left( \pi_U K_{k+1}, ... , \pi_U K_n\right).\end{align*}
\end{itemize}
(identity (ix) is deduced from (viii) by iteration).

Let $K \subset \R^n$ be a compact convex set. Its support function $h_K$, is defined on $\R^n$ by $h_K(x)=\max_{y\in K} \langle y,x\rangle$, where $\langle . , . \rangle$ denotes the usual scalar product on $\R^n$. Since $h_K(\lambda x)=\lambda h_K(x)$ for any $x\in \R^n$, $\lambda>0$, we shall more often consider $h_K$ as a function on $\mathbb{S}^{n-1}$. Note that $h_K$ characterizes $K$, since $K=\bigcap_u H^-(u,h_K(u))$, where $H^-(u,b)=\{z\in \R^n : \langle z,u\rangle \leq b\}$.

If we fix a convex body $K\subset \R^n$, then there exists a unique non-negative measure $S_K$ on $\mathbb{S}^{n-1}$, such that the following holds for any compact convex set $L$ :
\begin{equation}
\label{eqn:integralformula}
V_n(L,K[n-1])=\frac{1}{n} \int_{\mathbb{S}^{n-1}} h_L(u) dS_K(u).
\end{equation}
For instance, when $K=P$ is a polytope, the following integral representation of $V_n(L,P[n-1])$ is known :
\begin{equation}
\label{eq_polytope}
V_n(L, P[n-1])=\frac{1}{n} \sum_{u\in E(P)} h_L(u) |P^u|_{n-1}
\end{equation}
where $E(P)=\{\text{outer normal vectors of $P$}\}$ and $P^u=P\cap H(u,h_P(u))=\{y\in P : \langle y, u\rangle =h_P(u)\}$ is the facet whose outer normal vector is $u$. This means that $S_P$ is the discrete measure $S_P=\sum_{u\in E(P)} |P^u|_{n-1} \delta_u$, where $\delta_v$ denotes the Dirac measure at $v \in \mathbb{S}^{n-1}$.


Though the formula  \ref{eqn:integralformula} could be taken as a definition\footnote{the fact that knowing $\int h_K d\mu$ for all convex bodies $K$, is sufficient to characterize $\mu$, i.e. to know $\int f d\mu$ for any continuous function $f$ on the sphere, can be easily derived for instance from Lemma \ref{lemma:variationalmixed}} of the surface area measure $S_K$, one may alternatively first define $S_P$ for polytopes, via $S_P=\sum_{u\in E(P)} |P^u|_{n-1} \delta_u$, and then define $S_K$ for an arbitrary convex body $K$, by approximation \footnote{if $(P_k)$ is a sequence of polytopes approximating $K$, then the sequence of measures $(S_{P_k})$ is tight with respect to weak topology : define $S_K$ as the weak limit of $S_{P_k}$} (see \cite[Theorem 4.1.1, Theorem 4.2.1]{Sch1}). In this case, the integral formula \ref{eqn:integralformula} holds by definition for polytopes, and is deduced (in general) from continuity of mixed volumes, and of $(L\mapsto S_L)$.


Recall that if $\Omega$ is a closed subset of $\Sph$, and $g$ is a continuous function on $\Omega$, the Wulff-shape with respect to $(\Omega,g)$ is the convex body $W(\Omega,g)=\bigcap_{u\in\Omega} \{x\in \R^n : \langle x,u\rangle \leq g(u)\}$. Let $S_K$ be the surface area measure of $K$. More specifically, if $K$ is a convex body, $\Omega$ a closed subset of $\Sph$,  if $supp(S_K)\subset \Omega$ and if $f : \Omega \to \R$ is a continuous function, then we denote $\left (W_t\right)_t=\left(W(\Omega, h_K+tf)\right)_t$ the  family of Wulff-shape perturbations of $K$ associated with $(\Omega, f)$. Note that there exists $t_0=t_0(K)<0$ such that $V_n(W_t)>0$ for all $t >t_0$.

When $\Omega=\Sph$, we denote $W(g)=W(\Sph, g)$ the corresponding Wulff-shapes. See for instance \cite[Theorem 1.1]{SSZ2} where Wulff-shape perturbations (with $\Omega=\Sph$) were used to derive a characterization of $n$-simplices as the only convex bodies $K$ such that $G_K \geq 0$, where $G_K$ is the multi-linear form on $(\mathcal{K}^n)^n$ defined by $G_K(A_1, ... ,A_n)=V_n(A_1,K[n-1])V_n(K,A_2, ... ,A_n)-V_n(A_1, ... ,A_n)V_n(K)$.

 The following theorem is known as Alexandrov's variational lemma. We refer to \cite{Al1} for a proof, see also \cite[Lemma 7.4.3]{Sch1}. 
\begin{theorem}
Assume $K$ is a convex body, $supp(S_K)\subset \Omega$, and $f\in \mathcal{C}(\Omega, \R)$. For $t\in\R$, denote $W_t=W(\Omega, h_K+tf)$. Then $(t\mapsto V_n(W_t))$ is differentiable at $0$, and
\begin{equation}
 \label{eq_Aleksandrov_vol}
\diff{V_n(W_t)}{t}\bigg|_{t=0}=\lim_{t\to 0}\frac{V_n(W_t)-V_n(K)}{t}=\int_{\mathbb{S}^{n-1}}f(u)dS_K(u),
  \end{equation}
\end{theorem}
Minor modifications of the proof of the above theorem, yields a similar statement in terms of (first) mixed volumes, as follows.
\begin{lemma}[Alexandrov's variational lemma for mixed volume]
\label{lemma:variationalmixed}
Assume $K$ is a convex body, $supp(S_K)\subset \Omega$, and $f\in \mathcal{C}(\Omega, \R)$. Denote $W_t=W(\Omega, h_K+tf)$, $t\in \R$. Denote $V_1(t)=V_n(W_t,K[n-1])$. Then $(t\mapsto V_1(t))$ is differentiable\footnote{on both sides} at $0$, and :
\begin{equation} 
\label{eq_Aleksandrov}
\diff{V_1(t)}{t}\bigg|_{t=0}=\lim_{t\to 0}\frac{V_1(t)-V_n(K)}{t}=\frac{1}{n}\int_{\mathbb{S}^{n-1}}f(u)dS_K(u),
 \end{equation}
\end{lemma}

Fix $K, \Omega$ and $f$ (as above), and let $t_0=\sup \{t<0 : |W_t|_n=0\} <0$. Denote $(W_t)_t$ the associated family of Wulff-shape perturbations. One can easily check that for any $u\in \Sph$, the map $(t\mapsto h_{W_t}(u))$ is concave on $]t_0,+\infty[$. In particular, this map is both left and right-differentiable at $t=0$. In fact, Lemma \ref{lemma:variationalmixed} allows to draw a more precise conclusion here.
\begin{lemma}
\label{lemma:pointwiseCV}
Let $(W_t)_t$ be Wulff-shape perturbations of a given convex body $K$, with respect to $(\Omega, f)$. Then for $S_K$-almost every $u\in \Sph$:
\begin{equation} \label{pointwiseCV}
\diff{h_{W_t}(u)}{t}\bigg|_{t=0}=\lim_{t\to 0}\frac{h_{W_t}(u)-h_K(u)}{t}=f(u).
\end{equation}
\end{lemma}

We leave a proof of this pointwise convergence lemma in appendix, see also \cite[Theorem 4.1]{SSZ2} where the statement was derived from Alexandrov's variational lemma (Theorem \ref{eq_Aleksandrov_vol}). We will need Lemma \ref{lemma:pointwiseCV} below, for the proof of Proposition \ref{proposition:main}.


Finally, it will be convenient to introduce the following definition. 
\begin{definition}
For $K\in \conbod$, the \textit{Bezout constant} is given by
$$b_2(K)=\sup \frac{V_n(L_1, L_2, K[n-2]) V_n(K)}{V_n(L_1,K[n-1])V_n(L_2,K[n-1])},$$ where the supremum is over pairs of convex bodies $(L_1,L_2)\in \left(\conbod\right)^2$.
\end{definition}
Denote $F_{K,b}(A,B):=b V(A,K[n-1])V(B,K[n-1])-V(A,B,K[n-2]) V(K)$, then $b_2(K)$ can be equivalently defined as the least $b\geq 1$ such that $F_{K,b}\geq 0$. Notice that $b_2(TK)=b_2(K)$ for every affine transformation $T$. 

Clearly, $F_K=F_{K,1}.$ A Blaschke selection argument shows that the supremum is actually a maximum; in particular $b_2(K)<\infty$ for all $K\in\conbod.$ In fact\footnote{this is known as Fenchel inequality, see \cite{FGM}}, $b_2(K)\leq 2$ for any $K$. Notice that  $b_2(K) >1$ is equivalent to not having $F_K\geq 0$ (i.e. to existence of a pair $(A,B)$ such that $F_K(A,B)<0$).

\section{An excluding condition with isoperimetric ratios}
\label{sec:exclude}
If a property $\mathcal{P}$ (for instance, being decomposable) is such that when $K$ has $\mathcal{P}$, then $F_K$ cannot be non-negative (on all of $(\mathcal{K}^n)^2$), we shall say that $\mathcal{P}$ is an excluding condition (in the terminology of \cite{SZ}, a convex body $K$ cannot both satisfy $\mathcal{P}$ and satisfy \emph{Bezout inequalities}). It was shown in \cite{SSZ1,SSZ2} that being weakly decomposable (a property which in particular includes being a polytope other than an $n$-simplex, or being decomposable) is an excluding condition. Denote $\mathcal{K}_F$ the subclass of $\mathcal{K}^n$ consisting of convex bodies having at least one facet : $\mathcal{K}_F$ is closed under Minkowski addition, and contains the class of $n$-polytopes. In this section we give a new excluding condition, which concerns bodies $K\in \mathcal{K}_F$. We will work in $\R^n$ with $n\geq 3$.

Let $K\subset \R^n$ be a non-empty compact convex set. Recall there exists a unique affine subspace $H$ of $\R^n$, such that $K\subset H$, and $H$ has maximal (affine) co-dimension. The dimension of $K$ is defined as the dimension of this subspace $H$. Alternatively $\dim(K)$ can be defined as the maximal $k\geq 1$, such that one may find $k+1$ affinely independent points, within $K$.

Let $k\geq 2$ and let $K \in \mathcal{K}^n$ be $k$-dimensional.
Then denote $$\Isop(K)=\frac{1}{k}\frac{|\partial K|_{k-1}}{|K|_k}.$$

\begin{proposition}
\label{proposition:main}
Let $K$ be a convex body such that $K$ has a facet $F$ satisfying :
$\Isop(F)>\Isop(K)$. Then $b_2(K)>1$.
\end{proposition}
\begin{proof}
Assume $K$ has a facet $F=K^{u_0}$ such that $\frac{|\partial F|_{n-2}}{(n-1)|F|_{n-1}} > \frac{|\partial K|_{n-1}}{n|K|_n}$. Set $$c_0 := \frac{|\partial F|_{n-2}|K|_n}{n-1} - \frac{|\partial K|_{n-1}|F|_{n-1}}{n} >0,$$ and $c=\frac{2}{n-1}|\partial F|_{n-2}|K|_n > 2c_0$. Fix $\epsilon >0$ so that $c_0>\epsilon c$.

Since $S_K(\{u_0\})=|F|_{n-1} >0$ and $S_K(\mathbb{S}^{n-1})=|\partial K|_{n-1} <+\infty$, one may choose $f$ a non-negative and continuous function on the sphere, such that $f(u_0)=\max_{\mathbb{S}^{n-1}} f=1$ and $\int f(u) dS_K(u) < (1+\epsilon) S_K(u_0)=(1+\epsilon) |F|_{n-1}$. Fix such a positive function $f$, and define $L_t=W(h_K+tf)$, the Wulff shape with respect to function $h_K+tf$. One may think of $L_t$ as a perturbed version of $K$, with most of the perturbation in direction $u_0$.

\vspace{2mm}
Set $M=H_{u_0}^- \cap B_2^n=\{x\in B_2^n : \langle x,u_0\rangle \leq 0\}$ to be a half-euclidean ball, such that its unique facet is the euclidean ball $M^{u_0}=\pi_{u_0^{\perp}}(M)$ with $u_0$ as an outer normal vector. We will show that $F_K(A,B)<0$, for $A=L_t$, $B=M$, and $t>0$ is small enough, proving that $b_2(K)>1$.

 Assume $t\geq0$. Recall that the mixed surface area measure $\sigma:=S(M,K[n-2],.)$ is a non-negative measure, and that, since $f\geq 0$, $h_K(u) \leq h_{L_t}(u)$, for all $u\in \Sph$. It follows that :
\begin{align*}V_n(L_t,M,K[n-2])-V_n(M,K[n-1])&=\frac{1}{n}  \int (h_{L_t}-h_K)(u) d\sigma(u) 
\\
&\geq \frac{1}{n}  (h_{L_t}-h_K)(u_0) \sigma(\{u_0\}). \end{align*}
It follows from Lemma~\ref{lemma:pointwiseCV} that $h_{L_t}(u_0)-h_K(u_0)>t (1-\epsilon) f(u_0)=t (1-\epsilon)$, for $0<t<t_0(\epsilon)$. Also, $\sigma(\{u_0\})=S(M,K[n-2],u_0)=V_{n-1}(M^{u_0},K^{u_0}[n-2])=V_{n-1}(B_2^{n-1},F[n-2])=\frac{|\partial F|_{n-2}}{n-1}$.

Therefore, when $0<t<t_0(\epsilon)$ :
$$V_n(L_t,M,K[n-2])-V_n(M,K[n-1]) \geq (1-\epsilon) \frac{t}{n(n-1)} |\partial F|_{n-2}.$$
On the other hand, for any $t>0$ :
\begin{align*}V_n(L_t,K[n-1])-V_n(K)&=\frac{1}{n} \int (h_{L_t}-h_K)(u) dS_K(u) 
\\
&\leq \frac{t}{n} \int f(u)dS_K(u) < (1+\epsilon) \frac{t}{n} |F|_{n-1}.\end{align*}
where the last inequality is by the choice of $f$.
Also, by monotonicity of mixed volumes : 
$$V_n(M,K[n-1]) \leq V_n(B_2^n,K[n-1])=\frac{1}{n} |\partial K|_{n-1}.$$
We may now conclude by simple computations. 
For ease of notations, set $a_2=V_n(L_t,M,K[n-2])$, $a_0=V_n(K)$, $a_t=V_n(L_t,L[n-1])$ and $a_m=V_n(M,K[n-1])$. We shall show that $F_K(L_t,M)<0$, which rewrites as $(a_t-a_0)a_m - (a_2-a_m)a_0 <0$.

Denote $b_1=(a_2-a_m)a_0$ and $b_2=(a_t-a_0)a_m$, so that $F_K(L_t,M)=b_2-b_1$. The above lower and upper bounds give us : $b_1 \geq (1-\epsilon) \frac{t}{n(n-1)} |\partial F|_{n-2}|K|_n$ and $b_2\leq (1+\epsilon) \frac{t}{n^2} |\partial K|_{n-1}|F|_{n-1}$.

It follows that :
\begin{align*} F_K(L_t,M) &\leq  \epsilon \frac{t}{n} \left( \frac{|\partial F|_{n-2}|K|_n}{n-1} +  \frac{|F|_{n-1}|\partial K|_{n-1}}{n}\right)- \frac{t}{n} \left( \frac{|\partial F|_{n-2}|K|_n}{n-1} - \frac{|F|_{n-1}|\partial K|_{n-1}}{n}\right) 
\\
&\leq \frac{t}{n} (\epsilon c - c_0) <0 .
\end{align*}
\end{proof}

Recall that $F_K\geq 0$ is an affine-invariant property. On the other hand, the quantity
$$\sup_F \frac{\Isop(F)}{\Isop(K)}=\sup_{F} \frac{n |\partial F|_{n-2} |K|_n}{(n-1)|F|_{n-1} |\partial K|_{n-1}},$$ where the supremum is over the facets, is not affine-invariant. Thus, the above proposition immediately implies the following one.
\begin{proposition}
\label{p_facets_image}
Let $K$ be a convex body such that one of its affine pairs $K'=TK$ has a facet $F'$ satisfying :
$$\frac{|\partial F'|_{n-2}}{(n-1)|F'|_{n-1}} > \frac{|\partial K'|_{n-1}}{n|K'|_n}.$$
Then $b_2(K)>1$.
\end{proposition}
This yields that, if $\max_T \sup_F \frac{\Isop(TF)}{\Isop(TK)} >1$,
where the max is over $T\in O(n),$ then $b_2(K)>1.$ Let us list specific examples. Throughout, we will denote $|\cdot|$ for the Lebesgue measure (either of dimension $n$, $n-1$, or $n-2$).

\begin{itemize}
\item[a-] The unit cube has volume $1$, and so does any of its facet. Thus $|\partial C_n|=2n$, and $\Isop(C_n)=\frac{|\partial C_n|}{n|C_n|}=2$. Since each of its facets is a unit cube itself (of dimension $n-1$), they also satisfy $\Isop(F)=2$. Therefore $\max_F \frac{\Isop(F)}{\Isop(C)}=1$.
 But choose for $T$ the affine transform such that $Te_i=e_i$ for all $i\neq 1$, and $Te_1=2e_1$. Let $C'=TC$ be the resulting box.

\noindent Then, $\frac{|\partial F|}{(n-1)|F|}=2$ for the facet with outer normal $e_1$, while $\Isop(C')=2-\frac{1}{n}$. It follows that $\frac{\Isop(F)}{\Isop(C')}>1$, and hence $b_2(C)=b_2(C')>1$.

\item[b-] The octahedron. It has $2^n$ facets, each of them is a regular simplex of volume $|F|_{n-1}=\frac{\sqrt{n}}{(n-1)!}$.
Let $\Delta_{k}$ denote a regular $k$-simplex with edge length $\sqrt{2}$ (so $|\Delta_k|_k=\frac{\sqrt{k+1}}{k!}$). Then, if $n\geq 3$ :
 \begin{align*}
 \frac{\Isop(F)}{\Isop(O_n)}&=\frac{n|O_n||\partial F|}{(n-1)|\partial O_n| |F|}=\frac{n}{n-1} \frac{n |\Delta_{n-2}|}{n! |F|^2}
 \\
 &=\frac{n}{n-1}\frac{(n-1)!^2}{n! (n-2)!} \sqrt{n-1}=\sqrt{n-1} >1.
 \end{align*}
 \item[c-] The cylinders. Let $L$ be an $(n-1)$-dimensional convex body. Define $C=Conv(L,L')$ where $L'=L+te_n$ is a translate of $L$ parallel to it. Then $|C|=t|L|$, thus :
 $$
  \frac{\Isop(L)}{\Isop(C)}=\frac{n|C||\partial L|}{(n-1)|\partial C| |L|}=\frac{t n}{n-1} \frac{|\partial L|}{|\partial C|}=\frac{t n}{n-1} \frac{|\partial L|}{(2|L|+t|\partial L|)}.
 $$
 This ratio is greater than $1$ for large $t$ (as soon as $t>2\frac{n}{n-1}(\Isop(L))^{-1}$).
 Letting $T$ be an appropriate affine transform, i.e. $Te_n=t_0 e_n$, and $Te_i=e_i$ ($i\leq n-1$), where $t_0$ is large enough, one deduces that $b_2(C)=b_2(TC)>1$, because $\frac{\Isop(TL)}{\Isop(TC)}=\frac{\Isop(L)}{\Isop(TC)}>1$.

\item[d-] The half-ball. Denote $H_+=\{x_n\geq 0\}$ a closed half-space, $B_2^n$ the (unit) Euclidean ball, and $M=H_+ \cap B_2^n$ a half-ball. Then $|M|_n=\kappa_n /2$, $|\partial M|_{n-1}=\frac{n \kappa_n}{2}+\kappa_{n-1}$, and so $\Isop(M)=1+\frac{2\kappa_{n-1}}{n\kappa_n}$. The unique facet of $M$ is $F \approx B_2^{n-1}$, and so $\Isop(F)=1$. Hence $\frac{\Isop(F)}{\Isop(M)}<1$. 

Nonetheless, one can find an affine transform $T$ such that $\frac{\Isop(TF)}{\Isop(TM)}>1$, showing that $b_2(M)>1$.
Indeed, let $c>1$ be large enough and let $T$ be the unique affine transform such that $Te_n=ce_n$, and $Te_i=e_i$ for $i\leq n-1$. Then computations show that $\Isop(TM)=\frac{|\partial \mathcal{E}|}{n|\mathcal{E}|}+\frac{2\kappa_{n-1}}{n\kappa_n c}<\frac{|\partial \mathcal{E}|}{n|\mathcal{E}|}+\frac{1}{c} <1= \Isop(TF)=\Isop(F)$, showing $b_2(M)=b_2(TM)<1.$
(see appendix for computational details)

 \end{itemize}

In the above four examples (unit cube, octahedron $O_n$, cylinders, half-balls), we used Proposition \ref{proposition:main} (or Proposition \ref{p_facets_image}) to argue that $b_2(K)>1$. For these examples, the fact that $b_2(K)>1$ was already known : the cylinder, like the cube, is decomposable, the half-ball has (many) points of positive curvature on its boundary, and it was directly shown in [SZ15] that $b_2(O_n)=2$ (by taking well-chosen segments). The half-ball example suggests that $\max_F \frac{\Isop(F)}{\Isop(K)}$ could be minimal when $K$ is in its John's position.
 
 As a sanity check, one may wish to compute $\max_F \frac{\Isop(F)}{\Isop(\Delta)}$, when $\Delta$ is an $n$-simplex, to check this quantity is less than $1$. If $\Delta=T_n$ is a regular simplex with edge length $\sqrt{2}$, then $|T_n|=\frac{\sqrt{n+1}}{n!}$, and $|\partial T_n|=(n+1)|T_{n-1}|=(n+1) \frac{\sqrt{n}}{(n-1)!}$, so that $\Isop(T_n)=\sqrt{n(n+1)}$. It results that for a regular $n$-simplex :
 $$\max_F \frac{\Isop(F)}{\Isop(T_n)}=\frac{\sqrt{n(n-1)}}{\sqrt{n(n+1)}}=\sqrt{\frac{n-1}{n+1}}.$$

If $\Delta_n=Conv(0,e_1, ... , e_n)$, then $|\Delta_n|=\frac{1}{n!}$, and $|\partial \Delta_n|=\frac{n}{(n-1)!}+\frac{\sqrt{n}}{(n-1)!}$, hence $\Isop(\Delta_n)=n+\sqrt{n}$. There are two kinds of facets : one is $\Delta_{n-1}$, the other one is $T_{n-1}$. Their isoperimetric ratios are respectively $\Isop(T_{n-1})=\sqrt{n(n-1)}$ and $\Isop(\Delta_{n-1})=n-1+\sqrt{n-1}> \Isop(T_{n-1})$. Therefore :
 $$\max_F \frac{\Isop(F)}{\Isop(\Delta_n)}=\frac{\Isop(\Delta_{n-1})}{\Isop(\Delta_n)}=\frac{n-1+\sqrt{n-1}}{n+\sqrt{n}}.$$
 Note that $\max_F \frac{\Isop(F)}{\Isop(\Delta_n)} >\max_F \frac{\Isop(F)}{\Isop(T_n)}$. One may conjecture that $\max_F \frac{\Isop(F)}{\Isop(T_n)}=\min_T \max_F \frac{\Isop(TF)}{\Isop(T\Delta_n)}$, i.e. that the (maximal) ratio is minimal when the simplex is in its John's position.


Finally, we remark that Proposition~\ref{p_facets_image} yields a more concise proof of the following result, which states that having infinitely many facets, is an excluding condition.
\begin{theorem}[Theorem 4.2 in \cite{SSZ2}]
\label{theorem:infinite}
\label{infinitecond}
Let $K$ be a convex body with infinitely many facets. Then $b_2(K)>1$.
\end{theorem}

\begin{proof}
If $K$ has infinitely many facets, then infinitely many of them will satisfy $\Isop(F)>\Isop(K),$
so that Theorem \ref{theorem:infinite} follows from Proposition \ref{proposition:main}.

Let  $C$ be a $d$-dimensional convex body. Then 
$$\Isop(C)=\frac{1}{d}\frac{|\partial C|_{d-1}}{|C|_{d}}=\frac{1}{d} \frac{|\partial C|_{d-1}}{|C|_{d}^{d-1/d}}|C|_{d}^{-1/d} \geq \frac{1}{d}\frac{|\partial B_2^d|_{d-1}}{|B_2^d|_d^{d-1/d}}|C|_{d}^{-1/d}=\frac{\kappa_d^{1/d}}{|C|_{d}^{1/d}},$$ where we used the isoperimetric inequality, and where $\kappa_d$ denotes the volume of the $d$-dimensional euclidean ball.

It follows that $\Isop(F) \to +\infty$ when $|F|_{n-1} \to 0$. Since a convex body only has finite surface area measure, it follows that if $K$ has infinitely many facets, then all but finitely many of them will satisfy $\Isop(F)>\Isop(K)$.
\end{proof}

We conclude this section with the following question.

\noindent \emph{Question 1:}
Let $P$ be a polytope, other than a simplex. Denote $\mathcal{F}_{n-1}(P)$ the set of its facets. Do we necessarily have $\max_{T\in O(n)} \max_{F\in \mathcal{F}_{n-1}(P)} \frac{\Isop(F)}{\Isop(P)} >1 ?$

In words : (if $P$ is not a simplex) does there always exist an affine transform $T$ such that  $P'=TP$ has at least one facet $F'$ satisfying $\Isop(F')>\Isop(P')$ ?

\noindent If the answer is  positive, then it would yield another proof of Theorem  \ref{theorem:sszpolytopes} (by applying Proposition \ref{proposition:main}).

\section{Appendix}
\subsection{Wulff Shape Lemma}
We recall the statement of Alexandrov's variational lemma for mixed volumes, Lemma~\ref{lemma:variationalmixed}, and then provide a proof.
\begin{lemma}[Alexandrov's variational lemma, mixed volume version]
Assume $K$ is a convex body, $supp(S_K)\subset \Omega$, and $f\in \mathcal{C}(\Omega, \R)$. Denote $W_t=W(\Omega, h_K+tf)$, $t\in \R$. Denote $V_1(t)=V_n(W_t,K[n-1])$. Then $(t\mapsto V_1(t))$ is differentiable\footnote{on both sides} at $0$, and :
\begin{equation*} 
\diff{V_1(t)}{t}\bigg|_{t=0}=\lim_{t\to 0}\frac{V_1(t)-V_n(K)}{t}=\frac{1}{n}\int_{\mathbb{S}^{n-1}}f(u)dS_K(u),
 \end{equation*}
\end{lemma}

Let $K$ be a convex body, $\Omega$ a closed subset of $\Sph$ \emph{determining}\footnote{we are following Schneider's terminology \cite{Sch1}} $K$, meaning $K=\bigcap_{u\in \Omega} H^-(u,h_K(u))$, and let $f$ be a continuous \footnote{($f$ measurable and bounded on $\Omega$ would do as well)} function defined on $\Omega$. Recall we denote $W_t=W(\Omega, h_K+tf)$ the associated Wulff-shape perturbations : $W_t$ is a convex body if $t>t_0$.

Following Alexandrov's notations (see \cite{Al1}), we denote $V_k(t)=V_n(W_t[k], K[n-k])$ (in particular $V_0(t)=V_n(K)$ does not depend on $t$).
We prove $\frac{V_1(t)-V_0(t)}{t} \to \frac{1}{n} \int_{\Omega} f(u)dS_K(u)$ separately for $t\to 0^+$ then for $t\to 0^-$. For sake of clarity, we may omit some dependence from the notations : $V_1$ means $V_1(t)$, and similarly $V_k$ stands for $V_k(t)$.

\begin{proof}
For all $t$,  note that (by definition of $W_t$) $V_1-V_0=\frac{1}{n} \int_{\Omega} (h_{W_t}-h_K)(u) dS_K(u) \leq \frac{t}{n} \int_{\Omega} f(u) dS_K(u)$.
In particular, $\limsup_{t\to 0^+} \frac{V_1(t)-V_0}{t} \leq \frac{1}{n} \int_{\Omega} f(u)dS_K(u)$.

Let $\lambda_t=V_1(t)/V_0$, and $\mu_t=V_{n-1}(t)/V_n(t)$ (which are both well-defined and positive, as long as $t>t_0$). Then, using Brunn-Minkowski's inequality : 
\begin{equation*} \label{BMAl1}
(V_1-V_0) \sum_{r=0}^{n-1} \lambda_t^r=\frac{V_1^n-V_0^n}{V_0^{n-1}} \geq \frac{V_n V_0^{n-1}-V_0^n}{V_0^{n-1}}=V_n - V_0.
\tag{$\mathbf{BM1}$}
\end{equation*} And symmetrically :
 \begin{equation*} \label{BMAl2}
 (V_{n-1}-V_n) \sum_{r=0}^{n-1} \mu_t^r=\frac{V_{n-1}^n-V_n^n}{V_n^{n-1}} \geq \frac{V_0 V_n^{n-1}-V_n^n}{V_n^{n-1}}=V_0 - V_n.
 \tag{$\mathbf{BM2}$}
 \end{equation*}
 Therefore (combining the above two)
 $$V_1-V_0 \geq\alpha_t(V_n-V_{n-1}) \hspace{2mm} \text{ where } \alpha_t= \frac{ \sum_{r=0}^{n-1} \mu_t^r}{\sum_{r=0}^{n-1} \lambda_t^r}. $$
 Since $W_t \to K$ as $t\to 0$ (for the Hausdorff distance), continuity of $V_n(.)$ implies that $V_k(t)\to V_0$ (for all $k\leq n$), while (weak) continuity of $(L\mapsto S_L)$ implies that $S_{W_t} \to S_K$ (for the weak topology on Radon measures on the sphere). Therefore $\alpha_t \to 1$, while 
 \begin{align*}\frac{V_n-V_{n-1}}{t}&=\frac{1}{n} \int_{\Omega} \frac{(h_{W_t}-h_K)(u)}{t} dS_{W_t}(u)
 \\
 &=\frac{1}{n} \int_{\Omega} f(u) dS_{W_t}(u) \to \frac{1}{n} \int_{\Omega} f(u) dS_{K}(u)\end{align*}
 where we used that $h_{W_t}(u)=h_K(u)+tf(u)$ for $S_{W_t}$-a.e. $u\in \Omega$.
 Hence,
 $$\liminf_{t\to 0^+} \frac{V_1-V_0}{t} \geq \left(\lim_{t\to 0^+} \alpha_t\right)\left(\lim_{t\to0^+} \frac{V_n-V_{n-1}}{t}\right)=  \frac{1}{n} \int_{\Omega} f(u) dS_{K}(u).$$
 
 Similarly, for all $t_0<t<0$ : $\frac{V_1-V_0}{t} \geq  \frac{1}{n} \int_{\Omega} f(u) dS_{K}(u)$, while combining  \ref{BMAl1} and \ref{BMAl2} yields
  $$\limsup_{t\to 0^+} \frac{V_1-V_0}{t} \leq \left(\lim_{t\to 0^-} \alpha_t\right)\left(\lim_{t\to0^-} \frac{V_n-V_{n-1}}{t}\right)=  \frac{1}{n} \int_{\Omega} f(u) dS_{K}(u).$$
 \end{proof}

We recall the statement of the (almost-sure) pointwise convergence Lemma \ref{lemma:pointwiseCV} (which we used in the proof of Proposition \ref{proposition:main}), and provide a simple proof here.

\begin{lemma}
Let $(W_t)_t$ be Wulff-shape perturbations of a given convex body $K$, with respect to $(\Omega, f)$. Then for $S_K$-almost every $u\in \Sph$:
\begin{equation}
\diff{h_{W_t}(u)}{t}\bigg|_{t=0}=\lim_{t\to 0}\frac{h_{W_t}(u)-h_K(u)}{t}=f(u).
\end{equation}
\end{lemma}
\begin{proof}

First, recall that, for all $u\in \Sph$, the map $(t\mapsto h_{W_t}(u))$ is concave on $]t_0, +\infty[$, where $t_0=t_0(K,\Omega,f)=\inf \{t\in \R : |W_t|_n>0\}$.

Indeed, Let $t_0<t_1<t_2$, and let $\lambda \in (0,1)$. Let $x\in (1-\lambda)W_{t_1}+\lambda W_{t_2}$, and let $u\in \Omega$. Let $x_i\in W_{t_i}$, $i=1,2$, such that $x=(1-\lambda)x_1+\lambda x_2$.  Denote $t_{\lambda}:=(1-\lambda)t_1+\lambda t_2.$ Then
$$\langle x,u\rangle=(1-\lambda)\langle x_1,u\rangle+\lambda \langle x_2,u\rangle\leq (1-\lambda)h_{W_{t_1}}(u)+\lambda h_{W_{t_2}}(u) \leq h_K(u)+ t_{\lambda} f(u).$$
As this holds for any $u\in \Omega$, one concludes that $(1-\lambda)W_{t_1}+\lambda W_{t_2} \subset W_{t_{\lambda}},$
from which concavity (on the domain $]t_0,+\infty[$) of $(t\mapsto h_{W_t}(u))$ ,for any fixed $u\in \Sph$, readily follows.

In particular, for any $u\in \Sph$,  the map $(t\mapsto h_{W_t}(u))$ is concave on a neighborhood of $t=0$, implying that $\lim_{t\to 0^+} \frac{(h_{W_t}-h_K)(u)}{t}$ and $\lim_{t\to 0^-} \frac{(h_{W_t}-h_K)(u)}{t}$, exist for any $u\in \Sph$. We now explain how to deduce that right and left derivative (at $t=0$) coincide, for $S_K$-a.e. $u\in \Sph$, from lemma \ref{lemma:variationalmixed}.

We next prove that $\lim_{t\to 0^+}  \frac{(h_{W_t}-h_K)(u)}{t} =f(u)$, for $S_K$-almost every $u$ (the proof on the left of $0$ is similar).
Let $l_d(u)=\lim_{t\to 0^+}  \frac{(h_{W_t}-h_K)(u)}{t}$, which is well-defined for all $u\in \Sph$, by concavity. Note that (by definition of $W_t$), $l_d(u)\leq f(u)$, for all $u$. Let $U_{\epsilon}=\{u\in \Sph : l_d(u)<f(u)-\epsilon\}$. It is enough to show that $S_K(U_{\epsilon})=0$ (for arbitrary $\epsilon>0$).

By Fatou's lemma, 
\begin{align*}\limsup_{t\to 0^+} \int_{U_{\epsilon}}  \frac{(h_{W_t}-h_K)(u)}{t} dS_K(u) &\leq   \int_{U_{\epsilon}}  \limsup_{t\to 0^+}\frac{(h_{W_t}-h_K)(u)}{t} dS_K(u)
\\
&= \int_{U_{\epsilon}} l_d(u)dS_K(u) 
\\
&\leq \int_{U_{\epsilon}}  f(u)dS_K(u) -\epsilon S_K(U_{\epsilon}).\end{align*} Therefore :
$$\lim_{t\to 0^+} \frac{V_1(t)-V_0}{t} \leq \frac{1}{n} \int_{\Omega} f(u)dS_K(u) - \frac{\epsilon}{n} S_K(U_{\epsilon}).$$
It follows from lemma  \ref{lemma:variationalmixed} that $S_K(U_{\epsilon})=0$.

Similarly, let $l_g(u)=\lim_{t\to 0^-}  \frac{(h_{W_t}-h_K)(u)}{t}$, $u\in \Sph$, so that $l_g(u)\geq f(u)$ for all $u \in \Omega$, and let $V_{\epsilon}=\{u\in \Sph : l_g(u) > f(u)+\epsilon\}$. Fix $x_0\in int(K)$, and let $K'=K-x_0$. If $|t|$ is small enough, then a translate of $(1-C|t|)K'$ is contained in $W_t$, with $C=(\max_u |f|)/(\min_u h_{K'})>0$.

In other words, $\frac{(h_{W_t}-h_K)(u)}{t} \leq C h_{K'}(u)$ for all $u\in \Sph$, for any $t<0$ (with $|t|$ small enough). Hence, by dominated convergence,
\begin{align*} \lim_{t\to 0^-} \frac{V_1-V_0}{t}&=\frac{1}{n} \int_{V_{\epsilon}} l_g(u) dS_K(u) +\frac{1}{n} \int_{\Omega\setminus V_{\epsilon}} l_g(u) dS_K(u)
\\
& \geq \frac{1}{n} \int_{\Omega} f(u) dS_K(u) + \frac{\epsilon}{n} S_K(V_{\epsilon}).
\end{align*}
So that $S_K(V_{\epsilon})=0$, by lemma  \ref{lemma:variationalmixed}.
It follows that $l_g(u)=f(u)$ for $S_K$-a.e. $u\in \Omega$.
\end{proof}

\subsection{Examples when Proposition \ref{proposition:main} applies : computations}

We review one of the examples listed after Proposition \ref{p_facets_image}.

Let $M$ be half a Euclidean ball : $M=B_2^n \cap H_+$, with $H_+=\{x\in \R^n : x_n \geq 0\}$. Then
$$\Isop(M)=\frac{|\partial M|}{n|M|}=\frac{2|\partial M|}{n \kappa_n}=\frac{n\kappa_n}{n\kappa_n} + \frac{2\kappa_{n-1}}{n\kappa_n}=1+\frac{1}{nW_n} >1.$$ where $W_n=\int_0^{\pi/2} (\cos(\phi))^n d\phi$ (recall that $\kappa_n=2\kappa_{n-1} W_n$).

Let $T=T_a$ be the linear map such that $Te_i=e_i$ when $i\leq n-1$, and $Te_n=ae_n$. Then $TB_2^n=\mathcal{E}_a$ is an ellipsoid of volume $a\kappa_n$. Moreover one may compute its surface area :
\begin{figure}
  \includegraphics[scale=0.23]{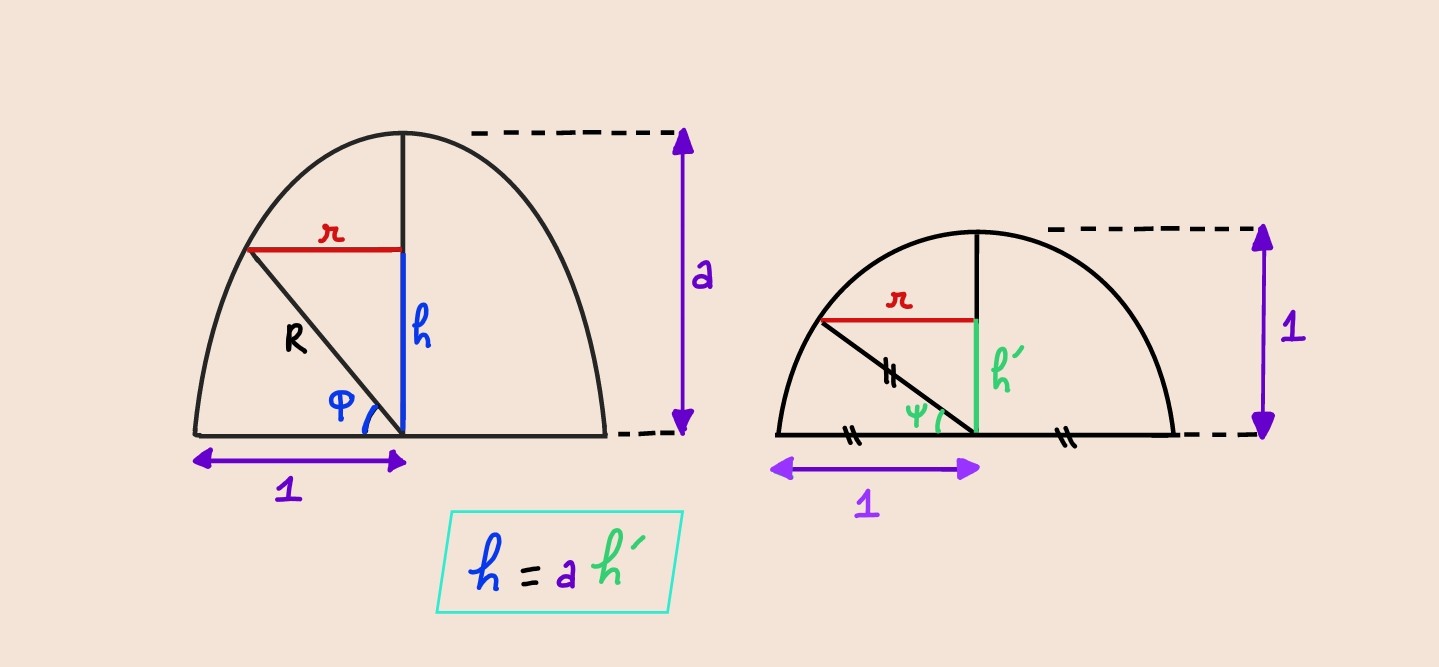} 
\end{figure}

$$|\partial \mathcal{E}_a|=2(n-1)\kappa_{n-1} \int_0^{\pi/2} R_{\phi}^{n-1} (\cos(\phi))^{n-2} d\phi=2(n-1)\kappa_{n-1} \int_0^{\pi/2} r_{\phi}^{n-1} \frac{d\phi}{\cos(\phi)}.$$

The equation of the surface $\partial \mathcal{E}_a$ is $r^2+\frac{h^2}{a^2}=1$, where $h=x_n$, and $r^2=x_1^2+...+x_{n-1}^2$.
The surface can be layered according to $r=R\cos(\phi) \in [0,1]$, and one may write $r$ as $r=:\cos(\psi)$ for some $\psi\in [0,\frac{\pi}{2}]$.
Then $\tan^2(\psi)=\frac{1-r^2}{r^2}=\frac{1}{a^2} \frac{h^2}{r^2}=\frac{1}{a^2} \tan^2(\phi)$, therefore (differentiating) $(1+\tan^2(\psi)) d\psi=\frac{1}{a}(1+\tan^2(\phi)) d\phi=\frac{1}{a}(1+ a^2\tan^2(\psi)) d\phi$, i.e. $d\phi=a \frac{1+\tan^2(\psi)}{1+a^2 \tan^2(\psi)} d\psi=a \frac{1}{1+(a^2-1)\sin^2(\psi)} d\psi$.

Also, $\cos(\phi)=\frac{r}{R}=\frac{r}{(a^2-(a^2-1)r^2)^{1/2}}=\frac{\cos(\psi)}{(1+(a^2-1)\sin^2(\psi))^{1/2}}$.

Changing the parametrization in $\phi$, into a parametrization in $\psi$, yields for the surface area :

$$|\partial \mathcal{E}_a|=2(n-1)\kappa_{n-1} \int_0^{\pi/2} (\cos(\psi))^{n-2} \frac{a}{(1+(a^2-1)\sin^2(\psi))^{1/2}} d\psi.$$

When $a=1$, one recovers $|\partial \mathcal{E}_1|=2(n-1)\kappa_{n-1} W_{n-2}=2\kappa_{n-1} nW_n=n\kappa_n$.
When $a>1$, one gets $|\partial \mathcal{E}_a|=\lambda (an\kappa_n)$, where $\lambda \in (a^{-1}, 1)$ is defined via :

$$\lambda :=\int_0^{\pi/2} (\cos(\psi))^{n-2} \frac{1}{(1+(a^2-1)\sin^2(\psi))^{1/2}} d\psi  \left(\int_0^{\pi/2} (\cos(\psi))^{n-2} d\psi\right)^{-1}.$$

Hence $TM$, the image of the half-ball $M$ under $T$, is a half-ellipsoid whose ``isoperimetric ratio'' is given by :

$$\Isop(TM)=\frac{|\partial \mathcal{E}_a|}{n |\mathcal{E}_a|}+\frac{2\kappa_{n-1}}{n|\mathcal{E}_a|}=\lambda +\frac{1}{a nW_n}.$$

The unique face $F$ of $TM$ satisfies $\Isop(F)=1$ : hence $TM$ satisfies the condition of Prop \ref{p_facets_image}, as soon as $a$ is such that :
$\lambda <1-\frac{1}{anW_n}$.

If $n=2$ and $a\geq 3$, then $1-\frac{1}{anW_n}=1-\frac{2}{a\pi}>\frac{3}{4}$, and one can easily check $\lambda <\frac{3}{4}$, so that $\Isop(TM)<1$ holds. Since $(\psi \mapsto (1+(a^2-1)\sin^2(\psi)))$ is increasing in $\psi \in [0,\pi/2]$, one may upper bound $\lambda=\lambda_a$ :
\begin{equation*}
\lambda=\frac{2}{\pi} \int_0^{\pi/2} \frac{1}{(1+(a^2-1)\sin^2(\psi))^{1/2}} d\psi< \frac{2}{\pi}  \left( \frac{\pi}{4}+ \frac{\pi}{4} \left( \frac{2}{a^2+1} \right)^{1/2}\right) < \frac{3}{4}
\hspace{3mm} (\text{if $a\geq 3$}).
\end{equation*}

If $n\geq 3$, then we use the estimate $\frac{\pi}{2(n+1)}<W_n^2<\frac{\pi}{2n}$, 
yielding \begin{align*}
nW_n> \sqrt{\frac{\pi n}{2}}\left(1+\frac{1}{n}\right)^{-1/2} > \sqrt{\frac{\pi n}{2}} \left(1-\frac{1}{2n}\right) \geq \frac{5}{6} \sqrt{\frac{\pi n}{2}} > \sqrt{n},\end{align*} and so $1-\frac{1}{anW_n}>1-\frac{1}{a\sqrt{n}}$.

 One may check that if $a\geq 4 n \sqrt{n}$, then $\lambda=\lambda_a < 1-\frac{1}{a\sqrt{n}}$.
  Indeed, letting $w_{\psi}:=\cos(\psi)^{n-2}$, similarly as in the $2$-dimensional case, one may upper bound $\lambda$ by splitting the integral :
  \begin{align*}\lambda W_{n-2}&=\int_0^{\pi/2n} (1+(a^2-1)\sin^2(\psi))^{-1/2} w_{\psi}d\psi+\int_{\pi/2n}^{\pi/2} (1+(a^2-1)\sin^2(\psi))^{-1/2} w_{\psi}d\psi 
  \\
  &< \frac{\pi}{2n}+ \frac{\pi}{2}   \frac{n}{(n^2+a^2-1)^{1/2}}\end{align*}
  using that $w_{\psi}=(\cos \psi)^{n-2}\leq 1$. Then one concludes, that for $a>n$ large enough :
  $$\lambda W_{n-2} <\left( 1-\frac{1}{a\sqrt{n}}\right) \sqrt{\frac{\pi}{2n}}<\left( 1-\frac{1}{a\sqrt{n}}\right) W_{n-2} $$
  so that $\Isop(TM)=\lambda+\frac{1}{anW_n}<\lambda+\frac{1}{a\sqrt{n}} <1=\Isop(F)$, as desired.

\end{document}